\documentclass[11pt]{article}

\usepackage[utf8]{inputenc}
\usepackage[T1]{fontenc}
\usepackage{amsmath, amsfonts, amssymb, amsthm}
\usepackage{graphicx,xcolor} 
\usepackage{subcaption}
\usepackage{url}

\newtheorem{theorem}{Theorem}[section]
\newtheorem{lemma}[theorem]{Lemma}
\newtheorem{remark}[theorem]{Remark}
\newtheorem{definition}[theorem]{Definition}
\newtheorem{claim}{Claim}

\begin{document}

\title{Cyclically $5$-edge-connected snarks with resistance~$2$ and flow resistance $n$}

\author{Davide Mattiolo${}^\dagger$, Pietro Negrini${}^{\dagger\ddagger}$, Silvia M.C. Pagani${}^\ddagger$\\[5pt]{\footnotesize
  ${}^\dagger$Department of Computer Science, KU Leuven Kulak, 8500 Kortrijk, Belgium} \\[5pt] {\footnotesize
  ${}^\ddagger$Department of Mathematics and Physics, Universit\`a Cattolica del Sacro Cuore,} \\ {\footnotesize 25133 Brescia, Italy}\\[5pt]
Emails: \texttt{davide.mattiolo@kuleuven.be},
\texttt{pietro.negrini@unicatt.it},\\
  \texttt{silvia.pagani@unicatt.it}
}
\maketitle

\begin{abstract}
Snarks are $2$-connected cubic graphs that do not admit a proper $3$-edge-coloring. 
For a cubic graph $G$, its resistance $r(G)$ is the minimum number of edges whose removal results in a $3$-edge-colorable graph, while its flow resistance $r_f(G)$ is the minimum number of edges whose removal results in a graph admitting a nowhere-zero $\mathbb{Z}_2 \times \mathbb{Z}_2$-flow.
In this paper, we provide an affirmative answer to a question recently posed by Allie, M\'a\v cajov\'a, and \v Skoviera by constructing a family of cyclically $5$-edge-connected snarks for which the ratio $r_f(G)/r(G)$ is arbitrarily large.
\end{abstract}

{\small 

}

\section{Introduction}

A \emph{snark} is a $2$-connected cubic (i.e.\ $3$-regular) graph not admitting a proper $3$-edge-coloring. 
Over the past decades, snarks attracted considerable interest primarily due to the fact that the study of many important conjectures, such as the Cycle Double Cover Conjecture \cite{PDseymour,Szekeres1973PolyhedralDO}, the 5-Flow Conjecture \cite{Tutte1954ACT}, and the Berge-Fulkerson Conjecture \cite{Fulkerson1971BlockingAA}, can be reduced to this subclass of graphs. It is indeed well-known that a smallest counterexample to such conjectures, if it exists, must be a snark. Therefore, investigating the structure of snarks is a central part of the study of these conjectures.

In order to better understand snarks, several authors introduced parameters measuring how far snarks are from being $3$-edge-colorable, see \cite{fms18}. 
In this paper we are interested in the interplay between two such parameters, namely, the resistance and the flow resistance whose definitions we recall below.
 
For a set $X\subseteq V(G)$ of a graph $G=(V(G),E(G))$, we define $\partial_G(X)$ to be the edge-cut induced by $X$, i.e.\ the set of edges of $G$ with exactly one end in $X$. 
An edge-cut $\partial_G(X)$ is \textit{cyclic} if both the subgraphs induced by $X$ and $V(G)\setminus X$ contain a cycle. A graph $G$ is called \textit{cyclically $k$-edge-connected} if it does not contain any cyclic edge-cut of cardinality less than $k$.

A \emph{proper $3$-edge-coloring} of a graph $G$ is a map $c\colon E(G)\to S$, for a suitable set $S$, with $|S|=3$, such that $c(e)\ne c(e')$ for every pair of adjacent edges $e,e'$ in $G$. A graph $G$ is called \emph{$3$-edge-colorable} if it admits a proper $3$-edge-coloring. The \textit{resistance} $r(G)$ \cite{steffen98, steffen04} of a cubic graph $G$ is the minimum number of edges whose removal leaves a $3$-edge-colorable graph.

Let $A$ be an additive abelian group with identity $0$. An \textit{$A$-flow} on a graph $G$ is a pair $(D,f)$, where $D$ is an orientation of $G$ and $f\colon E(G) \to A$ satisfies, for every vertex $v \in V(G), \sum_{e \in \partial^-(v)} f(e)= \sum_{e \in \partial^+(v)} f(e).$

The flow $(D,f)$ is called \emph{nowhere-zero} (or an $A$-NZF) if $f(e)\ne 0$ for every $e\in E(G).$



The \textit{flow resistance} \cite{fms18} of a cubic graph $G$, denoted by $r_f(G)$, is the smallest number of edges whose removal results in a graph admitting a nowhere-zero $\mathbb{Z}_2 \times \mathbb{Z}_2$-flow.


It is well-known that a $\mathbb{Z}_2 \times \mathbb{Z}_2$-NZF on a cubic graph $G$ defines a proper $3$-edge-coloring on $G$ and vice versa. Therefore, we will denote the three non-zero elements of the group $\mathbb{Z}_2 \times \mathbb{Z}_2$ by $a,b$ and $c$, and we will use them as colors in any $3$-edge-coloring we will define later in this article.

Thus,  
both $r(G)$ and $r_f(G)$ are zero if and only if $G$ is $3$-edge-colorable, and positive otherwise. In particular, it is well-known that $r(G)\ge2$ for every non-$3$-edge-colorable cubic graph $G$.
Despite their close relationship, $r(G)$ and $r_f(G)$ can diverge significantly. A conjecture by Fiol, Mazzuoccolo and Steffen~\cite{fms18} suggested that $r_f(G) \leq r(G)$ for every bridgeless cubic graph $G$. This was disproved in \cite{ams22} by providing a family of snarks satisfying $r_f(G) = 2r(G)$. More recently, Allie, M\' a\v cajov\'a, and \v Skoviera \cite{ams24} proved that the ratio $\frac{r_f(G)}{r(G)}$ can be arbitrarily large for cyclically $4$-edge-connected snarks and they asked whether this is also true under higher connectivity requirements.
The following theorem is the main result of this paper and answers in the positive to their problem.

\begin{theorem}\label{thm:main}
For every integer $n \geq 1$, there exists a cyclically $5$-edge-connected snark $H_n$ of order $40n+2$ such that $r(H_n) = 2$ and $r_f(H_n) = n$.
\end{theorem}

\section{Preliminaries}

In this section, we establish the basic definitions required for our construction. A \textit{semi-graph} $G = (V(G), E(G))$ consists of a finite set of vertices $V(G)$ and a multiset $E(G)$ of edges and semi-edges. An \textit{edge} is an unordered pair of distinct vertices $\{u, v\}$, denoted by $uv$, while a \textit{semi-edge} is a singleton $\{u\}$, denoted by $(u)$, representing an edge incident to only one vertex.
If $E(G)$ contains no semi-edges, then $G$ is a \textit{graph}.

We say that we \textit{join} two vertices $u, v \in V(G)$ by adding the edge $uv$ to $E(G)$. Similarly, we \textit{merge} two semi-edges $(u)$ and $(v)$ by removing them from $E(G)$ and adding the edge $uv$. Conversely, we define the operation of \textit{trimming} an edge $uv$ as removing the edge $uv$ and adding the two semi-edges $(u)$ and $(v)$. The \textit{removal} of a vertex $v$ adjacent to $u_1, u_2, \dots, u_n$ involves deleting $v$ and all incident edges $vu_i$; we denote the resulting semi-graph by $G-v$. More generally, for any subset of vertices $X$ of a semi-graph $G$, we denote by $G-X$ the semi-graph obtained by removing the vertices of $X$ one by one. Similarly, for any edge $e \in E(G)$, the notation $G-e$ represents the semi-graph obtained by removing the edge $e$ from $E(G)$. 

Let $G$ be a semi-graph. The \textit{degree} of a vertex $v$ in $G$, denoted by $\deg_G(v)$, is defined as the number of edges and semi-edges incident to it; $G$ is \textit{cubic} if every vertex has degree $3$.
We say that a semi-graph $G'$ is contained in $G$ if: (i) $V(G') \subseteq V(G)$, (ii) every edge $uv \in E(G')$ belongs to $E(G)$, and (iii) for every $v\in V(G')$, $\deg_{G'}(v)\le \deg_G(v)$. The \textit{distance} $d_G(u,v)$ between two vertices $u$ and $v$ in $G$ is defined as the number of edges in a shortest path connecting them. Similarly, we define the distance between a vertex $v$ and an edge $e = xy$ as the minimum of the distances from $v$ to the endpoints of~$e$.

The classical concepts of edge coloring and nowhere-zero flows for graphs extend naturally to semi-graphs. We recall that, given an edge-coloring of $G$, if $a$ and $b$ are two distinct colors, an $(a,b)$-\emph{Kempe chain} is a path in $G$ whose edges are colored alternately with colors $a$ and $b$. 
We recall here the well-known Parity Lemma in terms of semi-graphs. As a straightforward extension of the previous definition of edge-cut, when $G$ is a semi-graph and $X\subseteq V(G)$, $\partial_G(X)$ denotes the set of edges and semi-edges with exactly one endvertex in $X$.

\begin{lemma}[Parity Lemma \cite{descartes1948network}] \label{lem:parity}

Let $G$ be a cubic semi-graph and $f \colon E(G)\to S$ a proper $3$-edge-coloring of $G$. Then, for every $X\subseteq V(G)$ and $x\in S$, $|f^{-1}(x) |\equiv |\partial_G(X)|\mod{2}$.

\end{lemma}

The Parity Lemma naturally extends to flows on edge-cuts, leading to the following remark:

\begin{remark} \label{rmk:flow-edge-cut}

Let $\phi$ be a $\mathbb{Z}_2 \times \mathbb{Z}_2$-flow on a semi-graph $G$ and $X\subseteq V(G)$. Then, $\sum_{e \in \partial_G(X)}\phi(e) = 0$.
\end{remark}

\section{Construction of $H_n$}

We define the family of snarks $\{H_n\}_{n \in \mathbb{Z^+}}$ mentioned in Theorem~\ref{thm:main} in a recursive way. To achieve this, we first construct the semi-graph $Y_n$, which is in turn built from the semi-graph $Y_{n-1}$ and two copies of the semi-graph $Z$. To define $Z$, we first consider two specific semi-graphs, $M$ and $N$.

The semi-graph $M$ is obtained from the Petersen graph (which will be denoted by $P$) by removing the three vertices of a path of length two. We denote the five semi-edges of $M$ as $e_1, \dots, e_5$, as shown in Figure~\ref{poloM}. The semi-graph $N$ is instead constructed by removing a single vertex $v$ from $P$ and trimming an edge $e$ of its, such that $d_P(v,e)=2$. Due to the symmetry of the Petersen graph, this operation yields, up to isomorphism, the semi-graph $N$ shown in Figure \ref{poloN}, with the five semi-edges denoted as $f_1, \dots, f_5$.

\begin{figure}[htbp]
     \centering
     \begin{subfigure}[b]{0.48\textwidth}
         \centering
         \includegraphics[width=\textwidth]{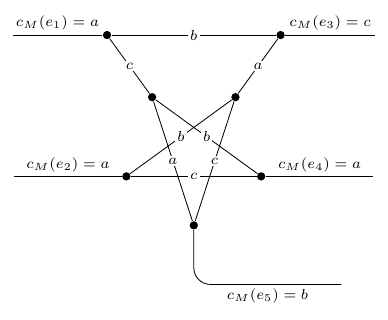}
         \caption{A proper $3$-edge-coloring $c_M$ on the semi-graph $M$.}
         \label{poloM}
     \end{subfigure}
     \hfill
     \begin{subfigure}[b]{0.48\textwidth}
         \centering
         \includegraphics[width=\textwidth]{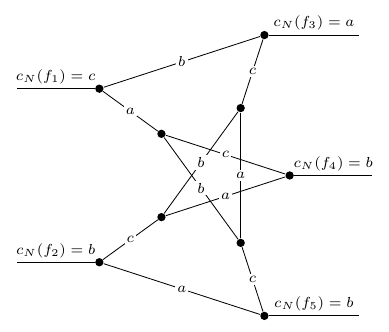} 
         \caption{A proper $3$-edge-coloring $c_N$ on the semi-graph $N$.}
         \label{poloN}
     \end{subfigure}
     \caption{Semi-graphs obtained from the Petersen graph.}
     
     \label{fig:semigraphs}
\end{figure}

\begin{lemma} \label{lem:properties_M}
The following properties hold for the semi-graph $M$:
\begin{enumerate}
    \item[$(i)$] $r(M) = 0$;
    \item[$(ii)$] for every $3$-edge-coloring $c_M$ of $M$ we have that either $c_M(e_1)=c_M(e_2)$ or $c_M(e_3)=c_M(e_4)$. 
\end{enumerate}
\end{lemma}

\begin{proof}
$(i)$ See Figure \ref{poloM} for a proper $3$-edge-coloring of $M$.

$(ii)$ First, by the Parity Lemma, it follows that three edges of the set $\{e_1, e_2, e_3, e_4, e_5\}$ must receive the same color, while the remaining two must receive the other two distinct colors. 

Suppose, by contradiction, that both edges of the pairs $\{e_1, e_2\}$ and $\{e_3, e_4\}$ receive different colors. In this case, one edge from the first pair, one from the second pair, and the edge $e_5$ would necessarily be assigned the same color. However, by joining $e_1$ and $e_2$ into a new vertex, and $e_3$ and $e_4$ into another new vertex, the $3$-edge-coloring of $M$ could be extended to a proper $3$-edge-coloring of the Petersen graph, a contradiction.
\end{proof}

\begin{lemma} \label{lem:properties_N}
The following properties hold for the semi-graph $N$:
\begin{enumerate}
    \item[$(i)$] $r(N) = 0$;
    \item[$(ii)$] for any proper $3$-edge-coloring $c_N$ of $N$, $c_N(f_1)\ne c_N(f_2)$. 
\end{enumerate}
\end{lemma}

\begin{proof}
$(i)$ See Figure \ref{poloN} for a proper $3$-edge-coloring of $N$.

$(ii)$ Let $c_N$ be a proper $3$-edge-coloring of $N$. Suppose by contradiction that $c_N(f_1)=c_N(f_2)$. By applying the Parity Lemma, it results that $\{c_N(f_3), c_N(f_4), c_N(f_5)\} = \{a,b,c\}.$ By joining $f_3, f_4,$ and $f_5$ in a new vertex and merging $f_1$ and $f_2$, the $3$-edge-coloring of $N$ could be extended to a proper $3$-edge-coloring of the Petersen graph, a contradiction.
\end{proof}

We now define the semi-graph $Z$. As shown in Figure~\ref{fig:poloZ_colorato}, $Z$ is constructed by merging the semi-edges $e_3$ and $e_4$ of $M$ with the edges $f_1$ and $f_2$ of $N$, respectively. Furthermore, we join the semi-edge $e_5$ of $M$ and the semi-edge $f_5$ of $N$ to a new vertex, from which the semi-edge $z_5$ of $Z$ originates. It is easy to see that $Z$ is still a cubic semi-graph. For the remaining semi-edges, we set $z_1:=e_1, z_2:=e_2, z_3:=f_3, z_4:=f_4$.

\begin{figure}[htbp]
    \centering
    \includegraphics[width=\textwidth]{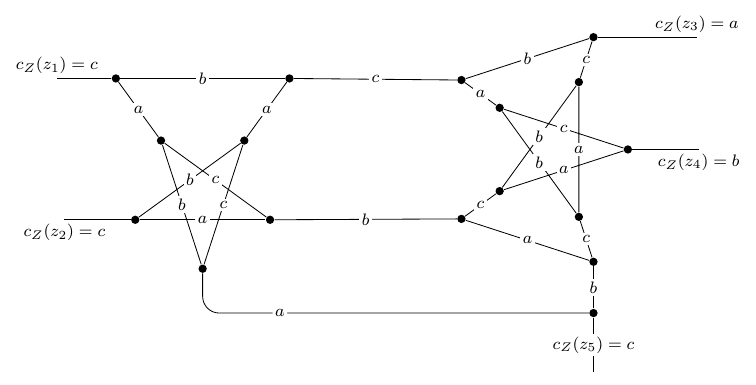}
    \caption{The proper $3$-edge-coloring $c_Z$ of the semi-graph $Z$.}
    \label{fig:poloZ_colorato}
\end{figure}

\begin{lemma} \label{lem:properties_Z}
The following properties hold for the semi-graph $Z$:
\begin{enumerate}
    \item[$(i)$] $r(Z) = 0$;
    \item[$(ii)$] in any proper $3$-edge-coloring $c'$ of $Z$, it results $c'(z_1)=c'(z_2)$ and $\{c'(z_3),c'(z_4), c'(z_5) \} = \{a,b,c\}$.
\end{enumerate}
\end{lemma}

\begin{proof}
$(i)$ See for example the proper $3$-edge-coloring $c_Z$ provided in Figure~\ref{fig:poloZ_colorato}.

$(ii)$ Since $c'(f_1)\neq c'(f_2)$ by Lemma \ref{lem:properties_N}, it follows that $c'(z_1)=c'(z_2)$ by Lemma \ref{lem:properties_M}. Then the Parity Lemma implies $\{c'(z_3), c'(z_4), c'(z_5) \} = \{a,b,c\}$.
\end{proof}

Note that Point $(ii)$ of the previous lemma was also proved in \cite[Section 5.8]{mrs2022}.


We now build the semi-graph $Y_1$. 


As illustrated in Figure~\ref{fig:Y1}, the semi-graph $Y_1$ is constructed using $Z$, $M$, and a copy of $Z$ denoted as $Z'$. For each semi-edge $z_i$ of $Z$, we denote the corresponding semi-edge of $Z'$ as $z_i'$. To obtain $Y_1$, we merge $z_3$ with $e_2$, $z_4$ with $e_1$, $e_4$ with $z_3'$, $e_3$ with $z_4'$, and $z_5$ with $z_5'$. We then relabel the five semi-edges of $Y_1$ as follows: we denote $z_1$ as $\alpha_1$, $z_2$ as $\beta_1$, $z_1'$ as $\gamma_1$, $z_2'$ as $\delta_1$, and $e_5$ as $\epsilon_1$. 
In Figure~\ref{fig:Y1}, we provide a representation of $Y_1$, where the component $Z'$ is drawn as a vertically reflected copy of $Z$.

\begin{figure}[htbp]
    \centering
    \includegraphics[width=\textwidth]{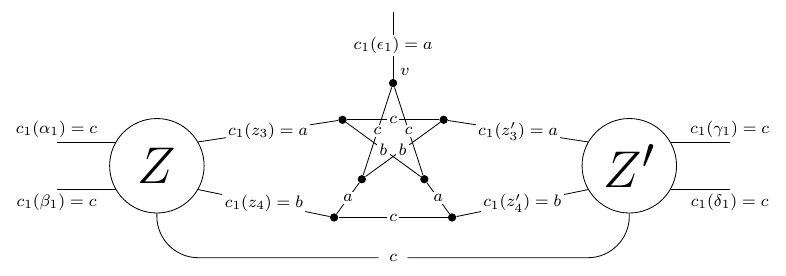} 
    \caption{A not proper $3$-edge-coloring $c_1$ of $Y_1$, with colors $a,b$ and $c$.}
    \label{fig:Y1}
\end{figure}

\begin{lemma} \label{lem:properties_Y1}
The following properties hold for the semi-graph $Y_1$:
\begin{enumerate}
    \item[$(i)$] $r(Y_1) = 1$;
    \item[$(ii)$] $r_f(Y_1) = 1$.
\end{enumerate}
\end{lemma}

\begin{proof}
$(i)$ By applying the proper $3$-edge-coloring $c_Z$ from Figure~\ref{fig:poloZ_colorato} to both $Z$ and $Z'$, and extending it to the remaining edges as shown in Figure~\ref{fig:Y1}, we obtain a not proper $3$-edge-coloring $c_1$ for $Y_1$. Under $c_1$, the only coloring conflict occurs at vertex $v$, proving that $r(Y_1) \leq 1$.

Suppose now, by contradiction, that there exists a proper $3$-edge-coloring $\phi$ of $Y_1$. According to Lemma~\ref{lem:properties_Z}$(ii)$, for the semi-graph $Z$, $\phi(z_3) \ne \phi(z_4)$, and for $Z'$, $\phi(z'_3) \ne \phi(z'_4)$. This violates Lemma~\ref{lem:properties_M}(ii), which states that at least one of these pairs must receive the same color in any proper $3$-edge-coloring. Consequently, $r(Y_1) = 1$. 

$(ii)$ Suppose that $r_f(Y_1) = 0$. This would imply that $Y_1$ admits a nowhere-zero $\mathbb{Z}_2 \times \mathbb{Z}_2$-flow and hence a proper $3$-edge-coloring, but this contradicts the first statement of this lemma.

Moreover, we can build a $\mathbb{Z}_2 \times \mathbb{Z}_2$-flow on $Y_1$ where the only edge receiving the zero value is the semi-edge $\epsilon_1$. This is done by taking the coloring $c_1$ from Figure~\ref{fig:Y1} and changing the value of $\epsilon_1$ from $a$ to $0$. Then $r_f(Y_1) = 1$.
\end{proof}

We can now construct the semi-graph $Y_i$ recursively for $i \ge 2$, starting from $Y_1$. We build $Y_i$ by taking $Y_{i-1}$, $Z$ and $Z'$, joined by adding 6 new vertices as shown in Figure~\ref{fig:Yi}. 

For $i \ge 2$, we relabel the five semi-edges of $Y_{i-1}$ as $\alpha_{i-1}, \beta_{i-1}, \gamma_{i-1}, \delta_{i-1}$, and $\epsilon_{i-1}$ in the way shown in Figure~\ref{fig:Yi}. In particular, to obtain $Y_i$, we join $z_3$ and $\gamma_{i-1}$ to a new vertex, $z_4$ and $\alpha_{i-1}$ to another new vertex, $z_3'$ and $\delta_{i-1}$ to a third new vertex, and $z_5$ and $z_5'$ to a fourth new vertex. Furthermore, we join $z_4'$ to a new vertex and $\beta_{i-1}$ to another new vertex. Finally, we restore the $3$-regularity as illustrated in Figure~\ref{fig:Yi}. 

\begin{lemma} \label{lem:properties_Yi}
For any positive integer $i$, the following properties hold for the semi-graph $Y_i$:
\begin{enumerate}
    \item[$(i)$] $r(Y_i) = 1$;
    \item[$(ii)$] $r_f(Y_i) = i$.
\end{enumerate}
\end{lemma}

\begin{proof}
$(i)$ Since $Y_1 \subset Y_i$ for every $i \ge 1$, we have $r(Y_i) \ge r(Y_1) = 1$ by Lemma \ref{lem:properties_Y1}. 
We now prove by induction that $Y_i$ has a $3$-edge-coloring $c_i$ such that the only conflict is on the vertex $v$ in $Y_1$ and that $c_i(\alpha_i) = c_i(\beta_i) = c_i(\gamma_i) = c_i(\delta_i) = c $ and $c_i(\epsilon_i) = a$. 
For $i=1$, we already showed that $c_1$ is a $3$-edge-coloring satisfying these requests.
For $i>1$, we define a $3$-edge-coloring $c_i$ on $Y_i$ having a single conflict at the same vertex $v$ in $Y_1$ as follows. First fix on $Y_{i-1}$ the $3$-edge-coloring $c_{i-1}$ (which leaves only one conflict in $v$ in $Y_1 \subset Y_{i-1}$). Then fix on $Z$ the proper $3$-edge-coloring $c_Z$ shown in Figure~\ref{fig:poloZ_colorato}. Moreover, let $c_{Z'}$ be the $3$-edge-coloring on $Z'$ defined as follows: consider the $3$-edge-coloring $c_Z$ on $Z'$ and the Kempe chain with colors $b$ and $c$ connecting $z_4$ with $z_5$; then $c_{Z'}$ is obtained by switching colors $b$ and $c$ along this Kempe chain. Finally, extend these colorings to the full semi-graph $Y_i$ as shown in Figure \ref{fig:Yi}. Thus, $r(Y_i) = 1$.

$(ii)$ We proceed again by induction to prove that $r_f(Y_i) = i$. By Lemma~\ref{lem:properties_Y1}, $r_f(Y_{1}) =1$. Let $i>1$ and suppose that $r_f(Y_{i-1}) = i-1$. Since $Y_{i-1} \subset Y_i$, it follows that $r_f(Y_i) \ge r_f(Y_{i-1}) = i-1$. 
Suppose by contradiction that $r_f(Y_i) = i-1$. Consider a $\mathbb{Z}_2 \times \mathbb{Z}_2$-flow $\phi$ on $Y_i$ with the minimum number of zeros. The $i-1$ edges with zero-value must all belong to $Y_{i-1}$; otherwise, the restriction of $\phi$ to $Y_{i-1}$ would yield a flow with fewer than $i-1$ zeros, contradicting the inductive hypothesis $r_f(Y_i) = i-1$. 
Consequently, all edges in $Y_i \setminus Y_{i-1}$ must have non-zero values under $\phi$. In this case, $\phi$ behaves as a proper $3$-edge-coloring on the edges outside $Y_{i-1}$. By Lemma~\ref{lem:properties_Z}, it results $\phi(\alpha_i) = \phi(\beta_i)$ and $\phi(\gamma_i) = \phi(\delta_i)$.
By Remark \ref{rmk:flow-edge-cut} and by the color impositions on the semi-graphs $Z$ and $Z'$, we have: $\phi(\alpha_i) + \phi(\beta_i) + \phi(\gamma_i) + \phi(\delta_i) + \phi(\epsilon_i) = \phi(\epsilon_i)= 0 $. This creates an $i$-th zero on $\epsilon_i$, and contradicts the assumption that only $i-1$ zeros were necessary. Thus, $r_f(Y_i) \ge i$. 
To prove that $r_f(Y_i) = i$, we observe that a $\mathbb{Z}_2 \times \mathbb{Z}_2$-flow with exactly $i$ zeros can be constructed as shown in Figure~\ref{fig:Yi_flow}, by giving the value $0$ to each edge $\epsilon_j$ for $j=1,\dots, i$.
\end{proof}

\begin{figure}[htbp]
    \centering
    \includegraphics[width=0.95\textwidth]{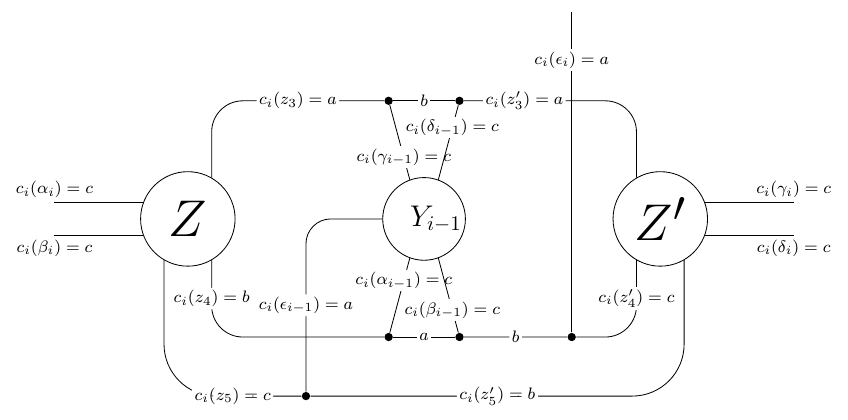}
    \caption{A (non-proper) $3$-edge-coloring $c_i$ of $Y_i$.}
    \label{fig:Yi}
\end{figure}

\begin{figure}[htbp]
    \centering
    \includegraphics[width=0.95\textwidth]{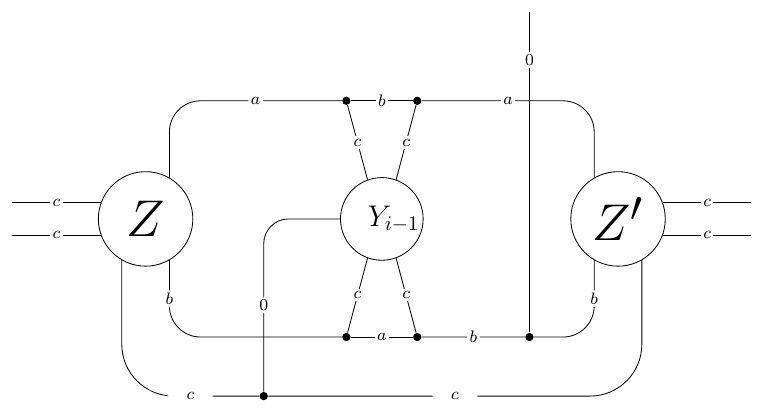}
    \caption{Inductive construction of a $\mathbb{Z}_2 \times \mathbb{Z}_2$-flow on $Y_i$ with exactly $i$ zero values.}
    \label{fig:Yi_flow}
\end{figure}

We can now construct the graph $H_n$ from the semi-graph $Y_n$ as follows: we merge the two horizontal semi-edges $\alpha_n$ and $\gamma_n$ to form a single edge $\pi$ and we connect the remaining three semi-edges $\beta_n$, $\delta_n$ and $\epsilon_n$ to a new vertex $x$, as shown Figure \ref{fig:Hn_color}.

We are now ready to prove the following.

\begin{theorem}\label{thm:res_flow_res_Hn}
For every integer $n \geq 1$, $H_n$ is a snark of order $40n +2$, with $r(H_n) = 2$ and $r_f(H_n) = n$.
\end{theorem}

\begin{proof}
By construction, $H_n$ is a cubic graph on $40n+2$ vertices.

Since $H_n$ contains $Y_n$ as a subgraph, it is not $3$-edge-colorable and so $r(H_n) \geq 2$. By extending the $3$-edge-coloring $c_n$ of $Y_n$ to the new edges as depicted in Figure~\ref{fig:Hn_color}, we find exactly two conflicts at vertices $v$ and $x$. Thus, $r(H_n) = 2$.

Regarding the flow resistance, since $Y_n \subseteq H_n$, it follows that $r_f(H_n) \geq r_f(Y_n) = n$. By using the $\mathbb{Z}_2 \times \mathbb{Z}_2$-flow on $Y_n$ shown in Figure~\ref{fig:Yi_flow}, we can define a $\mathbb{Z}_2 \times \mathbb{Z}_2$-flow on $H_n$ having only $n$ zero values. Thus, $r_f(H_n) = n$.
\end{proof}   

\begin{figure}[htbp]
    \centering
    \includegraphics[width=\textwidth]{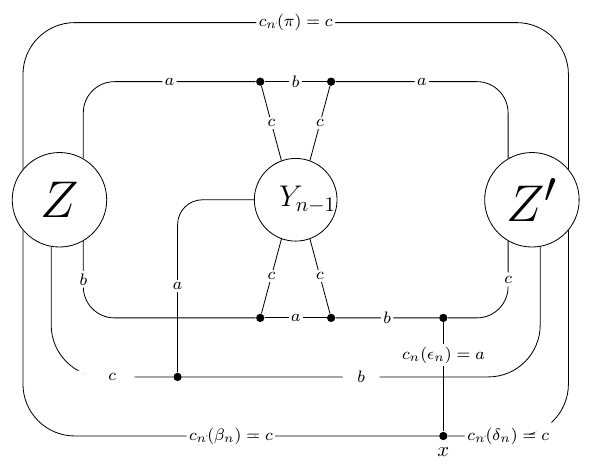}
    \caption{The $3$-edge-coloring $c_n$ of $H_n$.}
    \label{fig:Hn_color}
\end{figure}

\section{Cyclic connectivity of $H_n$}

In this section we prove that $H_n$ is cyclically $5$-edge-connected for every $n\ge1$.

\begin{definition}
Let $G_1, G_2$ be cubic graphs. For each $i \in \{1,2\}$, let $v_i \in V(G_i)$ and $e_i \in E(G_i)$ be such that $e_i$ is not incident to $v_i$. Moreover, let $G_i^* = G_i - v_i - e_i$. We define $G_1 \ast G_2 (v_1, e_1, v_2, e_2)$ as the set of cubic graphs obtained adding five edges connecting $G_1^*$ with $G_2^*$.
\end{definition}

If the vertices and edges $v_1, e_1, v_2, e_2$ in previous definition are clear from the context, we will denote the set $G_1 \ast G_2 (v_1, e_1, v_2, e_2)$ simply by $G_1 \ast G_2$.

Let $G \in G_1 \ast G_2 (v_1, e_1, v_2, e_2)$. We denote by $E_{conn}$ the set $\partial_G(G_1^*) = \partial_G(G_2^*)$. Edges in $E_{conn}$ are called \emph{connecting edges}. 

\begin{lemma}\label{lem:tree}
Let $T$ be a tree with only vertices of degree one and three. Let $\ell$ be the number of leaves of $T$.
Then $\ell=|V(T)|/2 + 1$.
\end{lemma}
\begin{proof}
The considered tree $T$ has $\ell$ vertices of degree one and $|V(T)|-\ell$ vertices of degree three. By applying the Handshaking Lemma, we have $2(|V(T)|-1) = 2|E(T)| = \sum_{v \in V(T)} \deg_T(v) = \ell + 3(|V(T)|-\ell)$. Solving for $\ell$ yields the statement.
\end{proof}

Using the previous lemma, one can prove the following fact. Recall that an edge-cut $\partial_G(X)$ of a graph $G$ is called \emph{trivial} if at least one of $X$ and $V(G)\setminus X$ contains only one vertex. It is called \emph{non-trivial} otherwise.

\begin{remark}\label{rmk:cubic_cuts}
In any connected cubic graph, every $1$-edge-cut, $2$-edge-cut, and non-trivial $3$-edge-cut is a cyclic cut.
\end{remark}

In the following lemma, we denote by $G\cap X$ the subgraph of $G$ induced by $V(G)\cap X$, for a given vertex set $X$.

\begin{lemma}\label{lem:main}
Let $G_1$ and $G_2$ be two cyclically $5$-edge-connected cubic graphs. Let $v_1 \in V(G_1), e_1 \in E(G_1)$ and $v_2 \in V(G_2), e_2 \in E(G_2)$. If $d_{G_1}(v_1, e_1) \ge 3$ and $d_{G_2}(v_2, e_2) \ge 2$, then every graph in the set $G_1 \ast G_2 (v_1, e_1, v_2, e_2)$ is cyclically $5$-edge-connected.
\end{lemma}

\begin{proof}
Let $G \in G_1 \ast G_2 (v_1, e_1, v_2, e_2)$. First of all, note that $G_i^*$ is a connected graph for $i\in\{1,2\}$.

Suppose, by contradiction, that there exists a cyclic $k$-edge-cut $\partial_G(X)$, for a suitable $X\subseteq V(G)$, with $k \in \{1,2,3,4\}$. Among all the possible cyclic $k$-edge-cuts, we choose one with $X$ of minimum cardinality.

\begin{claim}\label{claim1}
    For every $i\in \{1,2\}, X \not\subseteq V(G_i^*). $
\end{claim}

\begin{proof}
Suppose that $X \subseteq V(G_i^*)$. Since $\partial_G(X)$ is a cyclic $k$-edge-cut, $X$ must contain at least one cycle. 
From the fact that $k<5$ we deduce $X\subsetneq V(G_i^*)$. Moreover, since $G_i^*$ is connected, there is at least one edge with one endpoint in $G_i^*\cap X$ and the other in $G_i^* - X$.

Therefore, if $k\le3$, then $\partial_{G_i}(X)$ is an $h$-edge-cut in $G_i$, with $h=k$ if $k\le 2$, and $h\in \{k,k-2\}$ if $k=3$.
Note that $|V(G_i)\setminus X|\ge2$ and, since $X$ contains a cycle, we also have that $|X|\ge2$. Then, $\partial_{G_i}(X)$ is a non-trivial $h$-edge-cut in $G_i$ and so, by Remark \ref{rmk:cubic_cuts}, it is cyclic, a contradiction.

Set $k=4$. Note that $\partial_{G_i}(X)$ is an $h$-edge-cut, with $h\in\{2,4\}$. If $h=2$ the same argument as above applies and we get a contradiction. Therefore, let $h=4$. Since $G_i$ is cyclically $5$-edge-connected, $G_i-X$ cannot contain a cycle. Thus, by Lemma \ref{lem:tree}, $G_i-X$ consists of two vertices (see Figure \ref{fig:cases_claim1}). This means that $d_{G_i}(v_i,e_i)\le 1$, a contradiction. This proves the claim.
\end{proof}

\begin{figure}[htbp]
    \centering
    \includegraphics[scale=0.9]{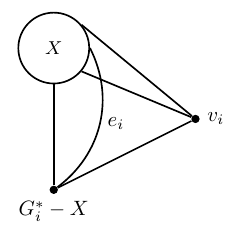}
    \caption{Structure of $G_i$ leading to a contradiction in Claim \ref{claim1}.}
    \label{fig:cases_claim1}
\end{figure}

In particular, $\partial_G(X)$ must contain at least one edge separating $G_1^* \cap X$ from $G_1^* - X$, and at least one edge separating $G_2^* \cap X$ from $G_2^* - X$.

\begin{claim}\label{claim2}
For $i \in \{1,2\}$, at least two edges of $\partial_G(X)$ must connect $G_i^* \cap X$ and $G_i^* - X$. 
\end{claim}

\begin{proof}
By contradiction, suppose there is only one edge $q \in \partial_G(X)$, which separates $G_i^* \cap X$ and $G_i^* - X$, for some $i\in \{1,2\}$. Let $m$ be the number of edges of $E_{conn}$ incident to $G_i^* \cap X$. See Figure \ref{fig:cases_claim2} for the cases $m\in \{3,4,5\}$.

    If $m\ge4$, then $G_i^*-X$ is adjacent to $c$ and to at most one more connecting edge. Therefore, $\partial_{G_i}(X\cup \{v_i\})$ is a non-trivial $h$-edge-cut in $G_i$, with $h\le 2$. Thus, by Remark \ref{rmk:cubic_cuts}, it is cyclic, a contradiction.
    
    If $m=3$, then $G_i^*-X$ is adjacent to $c$ and to two connecting edges. Then, $\partial_{G_i}(X\cup \{v_i\})$ is a non-trivial $h$-edge-cut in $G_i$, for $h\in \{1,3\}$. Since $G$ is a cyclically $5$-edge-connected graph, by Remark \ref{rmk:cubic_cuts} we deduce that $\partial_{G_i}(V(G_i)\setminus (X \cup \{v_i\}))$ is a trivial $3$-edge-cut. This means that $d_{G_i}(v_i,e_i)=1$, a contradiction.    
    
    If $m \in \{0,1,2\}$: this cases can be proved in the same way as previous points, arguing for $G_i^* \cap X$ instead of $G_i^* - X$.\qedhere
\end{proof}

\begin{figure}[htbp]
    \centering
    \begin{minipage}[b]{0.33\textwidth}
        \centering
        \includegraphics[width=\textwidth]{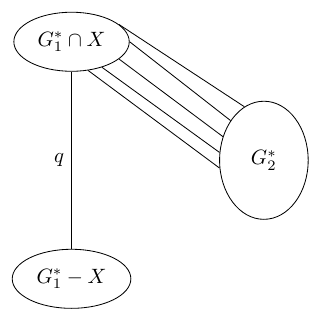} 
        \caption*{$m=5$}
    \end{minipage}\hfill
    \begin{minipage}[b]{0.33\textwidth}
        \centering
        \includegraphics[width=\textwidth]{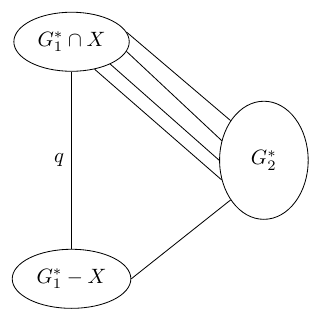} 
        \caption*{$m=4$}
    \end{minipage}\hfill
    \begin{minipage}[b]{0.33\textwidth}
        \centering
        \includegraphics[width=\textwidth]{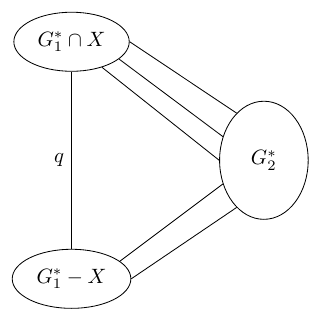} 
        \caption*{$m=3$}
    \end{minipage}
    \caption{Cases for $m \in \{3,4,5\}$ of Claim \ref{claim2}, for $i=1$.}
    \label{fig:cases_claim2}
\end{figure}

By Claims \ref{claim1} and \ref{claim2}, we deduce that $\partial_G(X)$ is a cyclic $4$-edge-cut.

Now, since $X$ was chosen to be of minimum cardinality, observe that the subgraph induced by $X$ must be connected. Hence, each connecting edge must connect either $G_1^*\cap X$ with $G_2^*\cap X$ or $G_1^*- X$ with $G_2^*- X$. Indeed, if there existed a connecting edge $f$ adjacent to $G_1^*\cap X$ and to $G_2^*- X$, then $f\in \partial_G(X)$ and hence $|\partial_G(X)|\ge5$, which contradicts the fact that $\partial_G(X)$ is a cyclic $4$-edge-cut.

Moreover, there exists at least one connecting edge between $G_1^*\cap X$ and $G_2^*\cap X$ and one connecting edge between $G_1^*- X$ and $G_2^*- X$, see Figure \ref{fig:quadranti}. If not, the restriction of $\partial_G(X)$ to $G_i$ would constitute a $2$-edge-cut in $G_i$, in contradiction with the connectivity assumption.

\begin{figure}[h]
    \centering
    \includegraphics[width=0.4\textwidth]{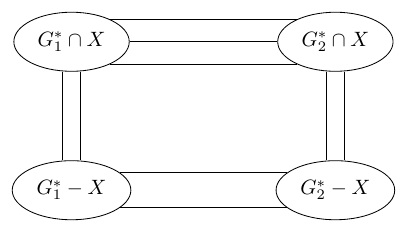} 
    \caption{A possible configuration of the graph $G$ with the considered set $X$.}
    \label{fig:quadranti}
\end{figure}

We are now ready to conclude the proof deriving a contradiction with the hypothesis on $G_1$.
Consider now $G_1^*$.
We have that $G_1^*\cap X$ is connected to $G_1^*- X$ by two edges of the $4$-edge-cut $C$, at least one connecting edge is incident to $G_1^*\cap X$ and at least one connecting edge to $G_1^*- X$.

If both $G_1^*\cap X$ and $G_1^*- X$ contain a cycle, by the same argument of Claim \ref{claim1}, it easy to see that $G_1$ would have a cyclic $k$-edge-cut, with $k<5$, a contradiction. Hence, at least one between $G_1^*\cap X$ and $G_1^*- X$ is a tree. 

Suppose now that $G_1^*\cap X$ contains a cycle and that $G_1^*- X$ is a tree.
Let $m$ be the number of connecting edges incident to $G_1^*\cap X$. If $m=1$ or $m=2$, we can see easily that $G_1$ would have a cyclic $3$ or $4$-cut, a contradiction. If $m=3$, $G_1$ has two possibilities: $(a)$ either $e_1$ connects two vertices of $G_1^*\cap X$ or $(b)$ $e_1$ connects a vertex of $G_1^*\cap X$ with a vertex of $G_1^*- X$, see Figure \ref{fig:G_1trees}.

In both cases, we arrive to a contradiction. In the case $(a)$, $G_1$ contains a triangle, and thus it is not cyclically $5$-edge-connected, a contradiction. In the case $(b)$, note that $d_{G_1}(v_1,e_1)\le2$ which also contradicts the hypothesis.

Hence $m=4$, which means that $V(G_1^*\setminus X)$ consists of only a vertex. Moreover $G_2^* \cap X$ is adjacent to four connecting edges, $G_2^* - X$ is adjacent to one connecting edge and, since $\partial_G(X)$ is a cyclic cut, $G_2^* - X$ contains a cycle. Then $G_2$ would have cyclic a $3$-cut, a contradiction.

\begin{figure}[htbp]
    \centering
    \begin{minipage}[t]{0.4\textwidth}
        \centering
        \includegraphics[width=\textwidth]{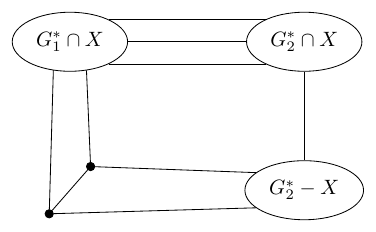} 
        \caption*{The graph $G$.}
    \end{minipage}\hfill
    \begin{minipage}[t]{0.3\textwidth}
        \centering
        \includegraphics[width=\textwidth]{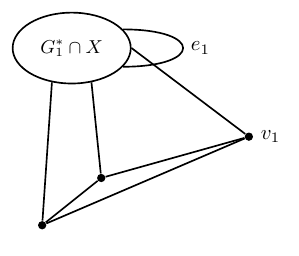} 
        \caption*{$G_1$ in case $(a)$.}
    \end{minipage}\hfill
    \begin{minipage}[t]{0.3\textwidth}
        \centering
        \includegraphics[width=\textwidth]{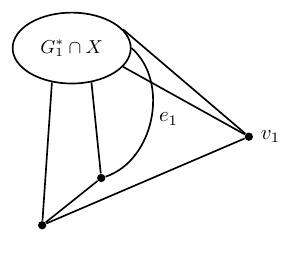} 
        \caption*{$G_1$ in case $(b)$.}
    \end{minipage}
    \caption{The case $m=3$, when $G_1^*\cap X$ contains a cycle and $G_1^*- X$ is a tree.}
    \label{fig:G_1trees}
\end{figure}

If $G_1^*\cap X$ is a tree and $G_1^*- X$ contains a cycle, then we reach the same conclusions. Hence both $G_1^*\cap X$ and $G_1^*- X$ are trees. By Lemma \ref{lem:tree}, this implies that $|V(G_1^*)|=5$ and $|V(G_1)|=6$. Then $G_1$ must be either the triangular prism or $K_{3,3}$, neither of which is cyclically $5$-edge-connected, in contradiction with our hypothesis. This completes the proof.
\end{proof}

Thanks to the previous lemma, we are now able to show the cyclic $5$-edge-connectivity of $H_{n+1}$ by induction on $n$, using the connectivity of $H_n$ and the operation defined above with the graph $J$ (see Figure \ref{fig:graph_J}). 

Indeed, if we trim $e_2$ in $J$ and remove the vertex $v_2$ in $J$, we obtain exactly the cubic semi-graph induced by $V(H_{n+1}) \setminus V(Y_n)$, for any integer $n \ge 1$. Furthermore, the semi-graph $Y_n$ is obtained from $H_n$ by splitting the edge $\pi$ and removing the vertex $x$. Consequently, $H_{n+1}\in H_n \ast J$.

For other methods to investigate cyclic edge-connectivity in cubic graphs we refer the interested reader to \cite{jozef,erik}.

\begin{figure}[h]
    \centering
    \includegraphics[width=1\textwidth]{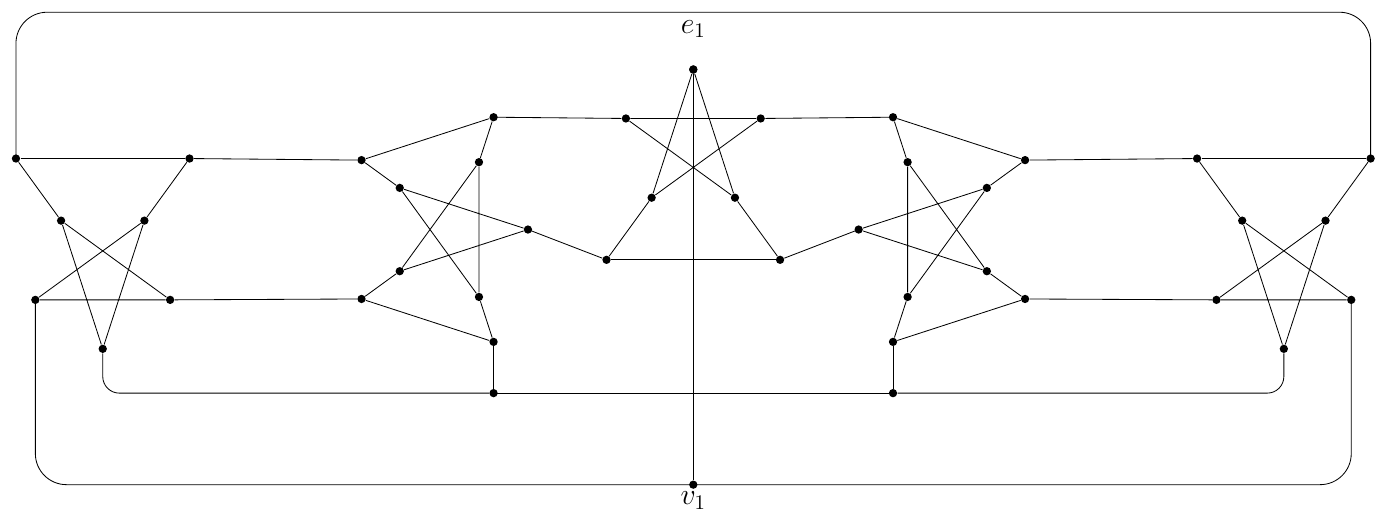} 
    \caption{The graph $H_1$.}
    \label{fig:H_1}
\end{figure}

\begin{figure}[h]
    \centering
    \includegraphics[width=1\textwidth]{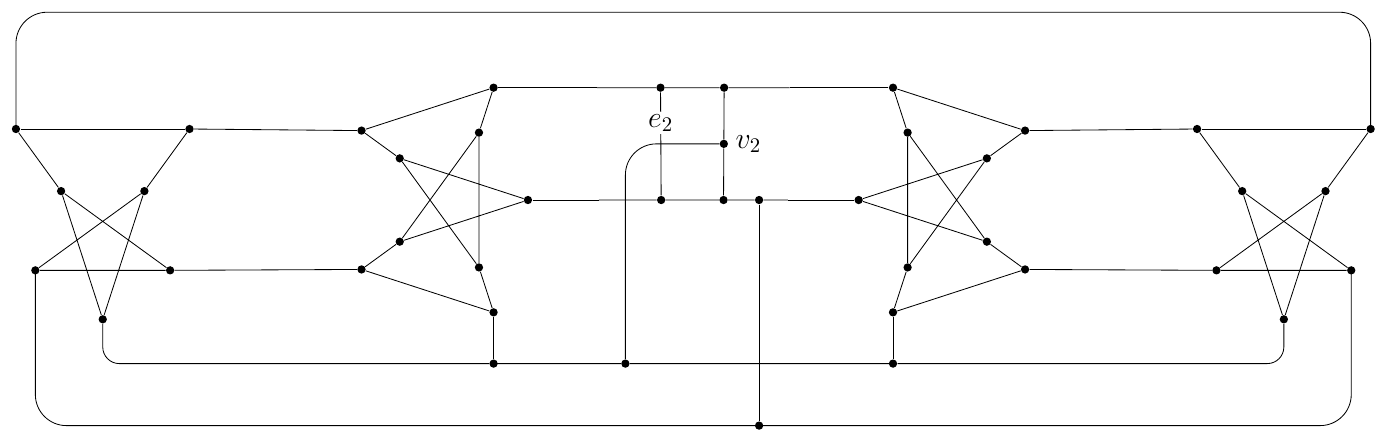} 
    \caption{The graph $J$.}
    \label{fig:graph_J}
\end{figure}

\begin{lemma} \label{lem: connectivity H_1}
$H_1$ and $J$ are cyclically $5$-edge-connected graphs.
\end{lemma}
\begin{proof}
It is not difficult to check that $H_1$ and $J$ are cyclically $5$-edge-connected cubic graphs (see Figure~\ref{fig:H_1} and Figure~\ref{fig:graph_J}).
\end{proof}

We are now ready to prove Theorem \ref{thm:main}.

\begin{proof}[Proof of Theorem \ref{thm:main}]
By Theorem \ref{thm:res_flow_res_Hn}, $r_f(H_n)=n$ and $r(H_n)=2$ for $n \in \mathbb{Z}^+$. We prove the cyclic connectivity by induction on $n$. According to Lemma \ref{lem: connectivity H_1}, the base case $H_1$ is cyclically $5$-edge-connected.

Assume now that $H_n$ is cyclically $5$-edge-connected. By Lemma \ref{lem: connectivity H_1}, the graph $J$ is also cyclically $5$-edge-connected. We observe that $d_J(e_2, v_2) = 2$ and $d_{H_n}(\pi, x) \ge 3$. Since $H_{n+1} \in H_n \ast J$, by Lemma \ref{lem:main} it follows that $H_{n+1}$ is cyclically $5$-edge-connected.
\end{proof}

\vspace{.5cm}
\textbf{Acknowledgments.} The second and the third author are members of INdAM-GNSAGA.

\bibliographystyle{plain}
\bibliography{references}

\end{document}